\documentclass[10pt]{amsart}
\usepackage{amssymb,amstext,amsmath,amscd,amsthm,amsfonts,enumerate,latexsym,stmaryrd,multicol,geometry,graphicx,mathrsfs,bm}
\usepackage[usenames]{color}
\usepackage[all]{xy}
\newcommand{\xdim}[1]{#1\text{-}\dim}
\newtheorem{thm}{Theorem}[section]
\newtheorem{lem}[thm]{Lemma}

\newtheorem{cor}[thm]{Corollary}
\theoremstyle{definition}
\newtheorem{dfn}[thm]{Definition}

\newtheorem{rem}[thm]{Remark}

\theoremstyle{remark}

\newtheorem*{claim*}{Claim}
\newtheorem*{ac}{Acknowledgments}
\newtheorem*{conv}{Convention}

\numberwithin{equation}{thm}
\def\A{\mathcal{A}}
\def\add{\operatorname{\mathsf{add}}}
\def\C{\mathcal{C}}
\def\cm{\mathsf{CM}}
\def\db{\operatorname{\mathsf{D^b}}}
\def\dim{\operatorname{dim}}
\def\End{\operatorname{End}}
\def\Ext{\operatorname{Ext}}
\def\G{\mathsf{G}}
\def\geq{\geqslant}
\def\H{\mathrm{H}}
\def\Hom{\operatorname{Hom}}
\def\inf{\operatorname{inf}}
\def\K{\mathrm{K}}
\def\leq{\leqslant}
\def\level{\operatorname{level}}
\def\Lotimes{\otimes^{\mathbf{L}}}
\def\m{\mathfrak{m}}
\def\mod{\operatorname{\mathsf{mod}}}
\def\P{\mathsf{P}}

\def\perf{\operatorname{\mathsf{perf}}}

\def\Rhom{\mathbf{R}\!\operatorname{Hom}}

\def\sup{\operatorname{sup}}
\def\T{\mathcal{T}}

\def\X{\mathcal{X}}

\def\Z{\mathbb{Z}}
\begin{document}
\title[A characterization of Cohen--Macaulay rings in terms of levels of perfect complexes]{A characterization of Cohen--Macaulay rings\\
in terms of levels of perfect complexes}
\author{Naoya Hiramatsu}
\address[N.H.]{Institute for the Advancement of Higher Education, Okayama University of Science,
1-1 Ridai-cho, Kita-ku, Okayama 700-0005, Japan}
\email{n-hiramatsu@ous.ac.jp} 
\author{Yuki Mifune}
\address[Y.M.]{Graduate School of Mathematics, Nagoya University, Furocho, Chikusaku, Nagoya 464-8602, Japan}
\email{yuki.mifune.c9@math.nagoya-u.ac.jp}
\author{Ryo Takahashi}
\address[R.T.]{Graduate School of Mathematics, Nagoya University, Furocho, Chikusaku, Nagoya 464-8602, Japan}
\email{takahashi@math.nagoya-u.ac.jp}
\urladdr{https://www.math.nagoya-u.ac.jp/~takahashi/}
\thanks{N.H. was partly supported by JSPS Grant-in-Aid for Scientific Research 21K03213 and 25K06966. Y.M. was partly supported by Grant-in-Aid for JSPS Fellows 25KJ1386. R.T. was partly supported by JSPS Grant-in-Aid for Scientific Research 23K03070}
\begin{abstract}
Let $R$ be a commutative noetherian ring, and let $C$ be a semidualizing $R$-module.
In this paper, we study levels of bounded complexes of finitely generated $R$-modules with respect to the full subcategory $\mathsf{G}_{C}(R)$ consisting of Gorenstein $C$-projective $R$-modules.
Our main result provides a characterization of the Cohen--Macaulayness of $R$ in terms of the finiteness of levels of perfect complexes with respect to $\mathsf{G}_{C}(R)$.
This recovers a recent theorem of Christensen, Kekkou, Lyle and Soto Levins on the Gorensteinness of $R$.
\end{abstract}
\maketitle
\section{Introduction}
The notion of levels of objects in a triangulated category has been introduced by Avramov, Buchweitz, Iyengar and Miller \cite{ABIM}.
A level measures the number of extensions required to build an object out of a given full subcategory of the triangulated category. 
This notion is closely related to the notion of strong generation in triangulated categories; see \cite{AT2015,AAITY,BV,J.C,DGI,Rou} for instance.
A natural setting for studying levels is the bounded derived category of an abelian category together with a full subcategory of the abelian category. 
To be more precise, for an abelian category $\A$, a full subcategory $\X$ of $\A$, and an object $M$ in the bounded derived category $\db(\A)$, one considers the level of $M$ with respect to $\X$, which is denoted by $\level_{\db(\A)}^{\X}M$. 
This problem has been studied by many authors.
Upper bounds for levels in terms of homological dimensions are given in \cite[Lemma 2.5.2]{ABIM} and \cite[Corollary 2.8]{AM} for instance.
Uniform bounds for levels were obtained in \cite[Theorem 4.1]{AAITY} in a general setting; see also Lemma \ref{AAITY_4.1}.
Lower bounds have been studied in \cite{AT2019,AGMSV,AM,M}.

For a commutative noetherian ring $R$, we denote by $\mod R$ the category of finitely generated $R$-modules, by $\P(R)$ the full subcategory of $\mod R$ consisting of projective $R$-modules, and by $\db(R)$ the bounded derived category of $\mod R$.
Avramov, Buchweitz, Iyengar and Miller \cite{ABIM} showed that if $R$ is a regular ring of finite Krull dimension $d$, then the level of every complex in $\db(R)$ with respect to $\P(R)$ is at most $d+1$. 
Conversely, if the levels of perfect complexes with respect to $\P(R)$ are uniformly bounded, then the ring $R$ is regular; see \cite[Proposition 7.25]{Rou}.
Motivated by this observation, Christensen, Kekkou, Lyle and Soto Levins \cite{CKLS} characterized Gorenstein rings in terms of levels of perfect complexes with respect to $\G(R)$, as follows.
Here, $\G(R)$ denotes the full subcategory of $\mod R$ consisting of Gorenstein projective $R$-modules, and $\perf(R)$ stands for the full subcategory of $\db(R)$ consisting of perfect $R$-complexes.
\begin{thm}[Christensen--Kekkou--Lyle--Soto Levins]\label{1}
Let $R$ be a commutative noetherian ring. Then the following two conditions are equivalent to each other.
\begin{enumerate}[\rm(1)]
\item
The ring $R$ is Gorenstein and $\dim R$ is finite.
\item
The number $l=\sup\{\level_{\db(R)}^{\G(R)}P\mid P\in\perf(R)\}$ is finite.
\end{enumerate}
When these equivalent conditions hold, one has an inequality $\dim R\leq l-1$.
\end{thm}
The main result of the present paper is the following theorem, where $\G_C(R)$ denotes the full subcategory of $\mod R$ consisting of Gorenstein $C$-projective $R$-modules.
\begin{thm}[Theorem \ref{main_thm}]\label{main_thm_int}
Let $R$ be a commutative noetherian ring.
Let $C$ be a semidualizing $R$-module.
Then the following two conditions are equivalent to each other.
\begin{enumerate}[\rm(1)]
\item
The ring $R$ is Cohen--Macaulay, $C$ is a canonical module of $R$, and $\dim R$ is finite.
\item
The number $l=\sup\{\level_{\db(R)}^{\G_{C}(R)}P\mid P\in\perf(R)\}$ is finite.
\end{enumerate}
When these equivalent conditions hold, one has an inequality $\dim R\leq l-1$.
\end{thm}
Theorem \ref{main_thm_int} shows that the finiteness of the levels of perfect complexes with respect to $\G_{C}(R)$ characterizes the Cohen--Macaulayness of $R$ with $C$ being a canonical module of $R$.
We should emphasize that Theorem \ref{main_thm_int} is not only a Cohen--Macaulay analog of Theorem \ref{1}, but also recovers Theorem \ref{1} by letting $C=R$.
\begin{conv}
Throughout this paper, we assume that all subcategories are strictly full, and that all complexes are cochain complexes.
Let $R$ be a commutative noetherian ring with identity.
\end{conv}
\section{Proof of the main theorem}
In this section, we recall some basic notions and prove the main result of this paper (Theorem \ref{main_thm}).
We begin with recalling the definition of levels in a triangulated category.
\begin{dfn}
Let $\C$ be an additive category.
Let $\T$ be a triangulated category.
\begin{enumerate}[\rm(1)]
\item
For a subcategory $\X$ of $\C$, we denote by $\add_{\mathcal{C}}\X$ the {\em additive closure} of $\X$ in $\mathcal{C}$, that is, the subcategory of $\mathcal{C}$ consisting of direct summands of finite direct sums of objects in $\X$.
\item
Let $\X$ be a subcategory of $\T$.
We set $\langle\X\rangle_{0}^{\T}=0$.
We denote by ${\langle\X\rangle}_{1}^{\T}$ the additive closure of the subcategory of $\T$ consisting of objects of the form $X[n]$, where $X\in\X$ and $n \in \mathbb{Z}$.
For an integer $r>0$, we inductively define $\langle\X\rangle_{r}^{\T}$ as the subcategory of $\T$ consisting of objects $M$ such that there exists an exact triangle $X\to Z\to Y\to\ $ in $\T$ with $X\in\langle\X\rangle_{r-1}^{\T}$, $Y\in\langle\X\rangle_{1}^{\T}$, and $M$ is a direct summand of $Z$.
\item
For a subcategory $\X$ of $\T$ and an object $M$ in $\T$, we define the $\X$-{\em level} of $M$ in $\T$, denoted by $\level_{\T}^{\X}M$, as the infimum of nonnegative integers $n$ such that $M\in\langle\X\rangle_{n}^{\T}$.
\end{enumerate}
\end{dfn}
We recall the definitions of semidualizing modules and Gorenstein $C$-projective modules.
\begin{dfn}
\begin{enumerate}[\rm(1)]
\item
A finitely generated $R$-module $C$ is said to be {\em semidualizing} provided that the natural homomorphism $R\to\End_{R}(C)$ is an isomorphism and $\Ext^{i}_{R}(C,C)=0$ for all positive integers $i$.
\item
Let $C$ be a semidualizing $R$-module, and let $(-)^\dag$ stand for the $C$-dual functor $\Hom_R(-,C)$.
A finitely generated $R$-module $M$ is called {\em Gorenstein $C$-projective} (or {\em totally $C$-reflexive}) if the natural homomorphism $M\to M^{\dagger\dagger}$ is an isomorphism and $\Ext_{R}^{i}(M\oplus M^{\dagger},C)=0$ for all positive integers $i$.
We denote by $\G_{C}(R)$ the subcategory of $\mod R$ consisting of Gorenstein $C$-projective $R$-modules.
\end{enumerate}
\end{dfn}
\begin{rem}
The ring $R$ is semidualizing as an $R$-module, and the subcategory $\G(R)=\G_{R}(R)$ consists of the finitely generated {\em Gorenstein projective} (or {\em totally reflexive}) $R$-modules.
If $R$ is a Cohen--Macaulay ring with a canonical module $\omega$, then $\omega$ is semidualizing; we adopt \cite[Definition 3.3.16]{BH} for the definition of a canonical module over a (not necessarily local) ring.
The subcategory $\cm(R)=\G_{\omega}(R)$ consists of the maximal Cohen--Macaulay $R$-modules.
\end{rem}
A {\em perfect} complex is by definition a bounded complex of finitely generated projective $R$-modules.
Denote by $\perf(R)$ the subcategory of $\db(R)$ consisting of perfect complexes.
For an $R$-module $C$, denote by $R\ltimes C$ the {\em idealization} of $C$ over $R$, that is, it is equal to $R\oplus C$ as an $R$-module, and has multiplication $(a,x)\cdot(b,y)=(ab,ay+bx)$ for $a,b\in R$ and $x,y\in C$.
For the details of idealizations, refer the reader to \cite[Exercise 3.3.22]{BH}.

A subcategory $\X$ of an abelian category $\A$ with enough projective objects is called {\em resolving} if $\X$ contains the projective objects of $\A$ and is closed under direct summands, extensions, and kernels of epimorphisms.
For an object $M$ in $\db(\A)$, we define the {\em $\X$(-resolution) dimension} of $M$, denoted by $\xdim{\X}M$, as the infimum of integers $n$ such that there exists a complex $X=\left(0\to X^{-n} \to X^{-n+1}\to\cdots\to X^{\sup M}\to0\right)$
with $X\cong M$ in $\db(\A)$ and $X^{i}\in\X$ for all $-n\leq i\leq\sup M:=\sup\{i\in\Z\mid\H^iM\ne0\}$.
By definition, $\xdim{\X}0=-\infty$.

The following lemma is an immediate consequence of \cite[Theorem 4.1]{AAITY} and \cite[Corollary 3.3]{M}.
\begin{lem}\label{AAITY_4.1}
\begin{enumerate}[\rm(1)]
\item
Let $\A$ be an abelian category with enough projective objects.
Let $\X$ be a resolving subcategory of $\A$.
Then for each object $M\in\db(\A)$, one has $\level_{\db(\A)}^{\X}M\leq\sup\{2,\xdim{\X}(\bigoplus_{i\in\mathbf{Z}}\H^{i}M)+1\}$.
\item
Let $R$ be a Cohen--Macaulay ring of finite Krull dimension $d$.
Then one has $\db(R)=\langle\cm(R)\rangle_{\max{\{2,d+1\}}}$.
In particular, there is an inequality $\sup\{\level_{\db(R)}^{\cm(R)}P\mid P\in\perf(R)\}\leq\max\{2,d+1\}$.
\end{enumerate}
\end{lem}
\begin{rem}
The Gorenstein case of Lemma \ref{AAITY_4.1}(2) is the same as \cite[Theorem 3.10]{CKLS}.
\end{rem}
We are now ready to state and prove the main result of this paper.
\begin{thm}\label{main_thm}
Let $C$ be a semidualizing $R$-module.
Then the following two conditions are equivalent.
\begin{enumerate}[\rm(1)]
\item
The ring $R$ is Cohen--Macaulay, $C$ is a canonical module of $R$, and $\dim R$ is finite.
\item
The number $l=\sup\{\level_{\db(R)}^{\G_{C}(R)}P\mid P\in\perf(R)\}$ is finite.
\end{enumerate}
When these equivalent conditions are satisfied, there is an inequality $\dim R\leq l-1$.
\end{thm}
\begin{proof} 
It follows from Lemma \ref{AAITY_4.1}(2) that the first condition implies the second.
Assume $l$ is finite, and pick a maximal ideal $\m$.
As the localization functor $(-)_{\m}:\perf(R)\to\perf(R_{\m})$ is dense and exact, it holds that
$$
\begin{array}{l}
\sup\{\level_{\db(R_\m)}^{\G_{C_\m}(R_{\m})}P\mid P\in\perf(R_{\m})\}
= \sup\{\level_{\db(R_{\m})}^{\G_{C_{\m}}(R_{\m})}Q_{\m}\mid Q\in\perf(R)\}\\
\phantom{\sup\{\level_{\db(R_{\m})}^{\G_{C_{\m}}(R_{\m})}P\mid P\in\perf(R_{\m})\}}
\leq \sup\{\level_{\db(R)}^{\G_{C}(R)}P\mid P\in\perf(R)\} =l<\infty.
\end{array}
$$
Replacing $R,C$ with the localizations $R_\m,C_\m$ at each maximal ideal $\m$ of $R$ respectively, we may assume that $R$ is local.
By assumption, $\perf(R)$ is contained in $\langle\G_{C}(R)\rangle_{l}^{\db(R)}$.
We begin with proving that
\begin{equation}\label{15r}
\perf(R)\otimes_{R}C:=\{P\otimes_RC\mid P\in\perf(R)\}
\subseteq\langle\G_{C}(R)\rangle_{l}^{\db(R)}.
\end{equation}
Let $P\in\perf(R)$.
Then there is an isomorphism $P\otimes_{R}C\cong\Hom_{R}(P^{*},C)$, where $(-)^{*}=\Hom_{R}(-,R)$.
Since $P^{*}\in\perf(R)\subseteq\langle\G_{C}(R)\rangle_{l}^{\db(R)}$, we have
$P\otimes_{R}C=\Rhom_{R}(P^*,C)\in\langle\G_{C}(R)\rangle_{l}^{\db(R)}$.
Thus \eqref{15r} follows.

Let $S=R\ltimes C$ be the idealization of $C$ over $R$, and let $I=0\oplus C$ be the ideal of $S$.
Then $S$ is a commutative noetherian local ring, and $S/I$ is isomorphic to $R$.
Since $I^2=0$, the ideal $I$ is an $S/I$-module.
We can view $C$ as an $S/I$-module via the isomorphism $S/I\cong R$.
Note then that the $S/I$-modules $C$ and $I$ are isomorphic to each other.
We shall prove that
\begin{equation}\label{16r}
\perf(S)\subseteq\langle\G(S)\rangle_{2l}^{\db(S)}.
\end{equation}
Let $P\in\perf(S)$.
Applying the functor $P\otimes_{S}-$ to the exact sequence $0\to I\to S\to S/I\to 0$ of $S$-modules, we get an exact sequence $0\to P\otimes_{S}I\to P\to P/IP\to 0$ of $S$-complexes, which induces an exact triangle
\begin{equation}\label{10r}
P\otimes_{S}I\to P\to P/IP \rightsquigarrow
\end{equation}
in $\db(S)$.
It is clear that $P/IP\in\perf(S/I)\subseteq\langle\G_{C}(S/I)\rangle_{l}^{\db(S/I)}$.
There are isomorphisms
$$
P\otimes_{S}I\cong P\otimes_{S}S/I\otimes_{S/I}I\cong P/IP\otimes_{S/I}I\cong P/IP\otimes_{S/I}C
$$
of $S/I$-complexes.
Using \eqref{15r}, we obtain $P/IP\otimes_{S/I}C\in\perf(S/I)\otimes_{S/I}C\subseteq\langle\G_{C}(S/I)\rangle_{l}^{\db(S/I)}$.
Hence the $S/I$-complex $P\otimes_SI$ belongs to $\langle\G_{C}(S/I)\rangle_{l}^{\db(S/I)}$.
Consider the exact functor
$$
\Phi =-\Lotimes_{S/I}M : \db(S/I)\to\db(S),
$$ 
where $M={}_{S/I}(S/I)_S$ is the $(S/I,S)$-bimodule $S/I$.
By \cite[Proposition 2.13(2)]{HJ} we see that $\Phi(\G_{C}(S/I))$ is contained in $\G(S)$, so that $\Phi(\langle\G_{C}(S/I)\rangle_{l}^{\db(S/I)})$ is contained in $\langle\G(S)\rangle_{l}^{\db(S)}$.
Hence both of the $S$-complexes $P\otimes_{S}I$ and $P/IP$ belong to $\langle\G(S)\rangle_{l}^{\db(S)}$.
It follows from \eqref{10r} that $P\in\langle\G(S)\rangle_{2l}^{\db(S)}$.
Thus \eqref{16r} follows.

By \cite[Theorem 2.6]{CKLS} and \eqref{16r}, the ring $S$ is Gorenstein.
The proof of \cite[Theorem (7)]{Rei} shows that $R$ is Cohen--Macaulay and $C$ is a canonical module of $R$.
Let $\bm{x}=x_{1},\ldots,x_{d}$ be a system of parameters of $R$, and let $K=\K(\bm{x},R)$ be the Koszul complex of $\bm{x}$.
By \cite[page 68(8)]{G} and \cite[Corollary 5.5(3)]{M}, one has
$l\geq\level_{\db(R)}^{\G_{C}(R)}K=\level_{\db(R)}^{\G_{C}(R)}R/(\bm{x})\geq\mathrm{G}_C\operatorname{-dim}_R R/(\bm{x})+1=d+1$.
We are done.
\end{proof}
As an immediate consequence of Theorem \ref{main_thm}, we get a characterization of the Cohen--Macaulay rings $R$ with canonical modules in terms of the finiteness of the $\G_{C}(R)$-levels of perfect $R$-complexes.
\begin{cor}
The following two conditions are equivalent to each other.
\begin{enumerate}[\rm(1)]
\item
The ring $R$ is Cohen--Macaulay, admits a canonical module, and has finite Krull dimension.
\item
The number $\inf\{\sup\{\level_{\db(R)}^{\G_C(R)} P\mid P\in\perf(R)\}\mid C \text{ is a semidualizing $R$-module}\}$ is finite.
\end{enumerate}
\end{cor}
Letting $C=R$ in Theorem \ref{main_thm}, we recover the following result proved in \cite[Theorems 2.10 and 3.10]{CKLS}.
\begin{cor}[Christensen--Kekkou--Lyle--Soto Levins]
The following conditions are equivalent.
\begin{enumerate}[\rm(1)]
\item
The ring $R$ is Gorenstein of finite Krull dimension.
\item
The number $l=\sup\{\level_{\db(R)}^{\G(R)}P\mid P\in\perf(R)\}$ is finite.
\end{enumerate}
When these equivalent conditions are satisfied, the Krull dimension of $R$ is at most $l-1$.
\end{cor}
\begin{ac}
The authors thank Takuma Aihara for helpful comments on trivial extensions.
\end{ac}

\end{document}